\documentclass[reqno,11pt]{amsart}
\usepackage{a4wide,color,eucal,enumerate,mathrsfs}
\usepackage[normalem]{ulem}
\usepackage{amsmath,amssymb,epsfig,amsthm} 
\usepackage[latin1]{inputenc}
\usepackage{psfrag}
\usepackage{graphicx}
\usepackage{transparent}
\usepackage{import}
\def\az{\alpha}
\def\a{\mathbf{a}}
\def\v{\mathbf{v}}

\def\dist{{\mathop\mathrm{\,dist}}}
\def\sdist{{\mathop\mathrm{\,dist_{s}}}}

\def\ez{\epsilon}

\def\bint{{\ifinner\rlap{\bf\kern.35em--}
\int\else\rlap{\bf\kern.45em--}\int\fi}\ignorespaces}

\def\bbint{{\ifinner\rlap{\bf\kern.35em--}
\hspace{0.078cm}\int\else\rlap{\bf\kern.45em--}\int\fi}\ignorespaces}

\newcommand{\R}{\mathbb{R}}

\newtheorem{thm}{Theorem}[section]
\newtheorem{lem}[thm]{Lemma}
\newtheorem{cor}[thm]{Corollary}
\numberwithin{equation}{section}

\theoremstyle{remark}

\def\bint{{\ifinner\rlap{\bf\kern.35em--}
\int\else\rlap{\bf\kern.45em--}\int\fi}\ignorespaces}

\usepackage{graphicx}
\usepackage{import}
\usepackage{xifthen}
\usepackage{pdfpages}
\usepackage{transparent}

\title[Strong stability of crystals]{Strong stability for the Wulff inequality\\ with a crystalline norm}
\author{Alessio Figalli, Yi Ru-Ya Zhang}
\date{\today}
\address{ETH Z\"urich, Department of Mathematics, R\"amistrasse 101, 8092, Z\"urich, Switzerland}
\email{alessio.figalli@math.ethz.ch}  
\email{yizhang3@ethz.ch}

\thanks{Both authors have received funding from the European Research Council under the Grant Agreement No.
721675 ``Regularity and Stability in Partial Differential Equations (RSPDE)''}
\begin{document}
\maketitle

\begin{abstract}
Let $K$ be a convex polyhedron and $\mathscr F$ its Wulff energy, and let $\mathscr C(K)$ denote the set of convex polyhedra close to $K$ whose faces are parallel to those of $K$. 
We show that, for sufficiently small $\ez$, all $\ez$-minimizers belong to $\mathscr C(K)$.

As a consequence of this result we obtain the following sharp stability inequality for crystalline norms: 
There exist $\gamma=\gamma(K,n)>0$ and $\sigma=\sigma(K,n)>0$ such that,
whenever $|E|=|K|$ and $|E\Delta K|\le \sigma$ then
 $$
 \mathscr F(E) - \mathscr F(K^\a)\ge \gamma |E \Delta K^\a| \qquad \text{for some }K^\a \in \mathscr C(K).
 $$
In other words, the Wulff energy $\mathscr F$ grows very fast (with power $1$) away from the set $\mathscr C(K).$
The set $K^\a \in \mathscr C(K)$ appearing in the formula above can be informally thought as a sort of ``projection'' of $E$ on the set $\mathscr C(K).$

Another corollary of our result is a very strong rigidity result for crystals: For crystalline surface tensions, minimizers of $\mathscr F(E)+\int_E g$ with small mass are polyhedra with sides parallel to the ones of $K$. In other words, for small mass, the potential energy cannot destroy the crystalline structure of minimizers. This extends to arbitrary dimensions a two-dimensional result obtained in \cite{FM2011}.
\end{abstract}

\tableofcontents

\section{Introduction}
The Wulff construction \cite{W1901} gives a way to determine the equilibrium shape of a droplet or crystal of fixed volume inside a separate phase (usually its saturated solution or vapor). Energy minimization arguments are used to show that certain crystal planes are preferred over others, giving the crystal its shape. 

The anisotropic surface energy is a natural choice in this circumstance. 
Given a  convex, positive, $1$-homogeneous function $f : \mathbb R^n \to [0,+\infty)$,  we define the anisotropic surface energy for a set of finite perimeter $E \subset \mathbb R^n$ as
$$\mathscr F(E)=\int_{\partial^* E} f(\nu_E) \, d\mathcal H^{n-1}, $$
where $\nu_E$ is the measure-theoretic outer unit normal to $E$, and $\partial^* E$ is its reduced boundary. We call $f$ the surface tension for $\mathscr F$.
Observe that when $f(\nu)=|\nu|$, we obtain the notion of classical perimeter.

 Every volume-constrained minimizer  of the surface energy $\mathscr F$ is obtained by translation and dilation of a bounded open convex set $K$, called the {\it Wulff shape} of $f$. When $f(\nu)=|\nu|$, $K$ is exactly the unit ball. 
 Also $K$ can be equivalently characterized by
$$K=\bigcap_{\nu\in \mathbb S^{n-1}}\left\{x\in \mathbb R^n\colon x\cdot \nu<f(\nu)\right\}=\left\{x\in \mathbb R^n\colon f_*(x) <1 \right\}.$$
Here $f_*\colon \mathbb R^n\to [0,+\infty)$ is the dual of $f$ defined as
$$f_*(x)=\sup\{x\cdot y\colon f(y)=1\},\quad x\in \mathbb R^n.$$
The {\it Wulff inequality} states that, for any set of finite perimeter $E\subset \mathbb R^n$, one has
$$\mathscr F(E)\ge n |K|^{\frac 1 n}|E|^{\frac {n-1} n}, $$
see e.g.\ \cite{T1978, MS1986}. 
In particular,
$$\mathscr F(E)\ge \mathscr F(K)=n|K|$$
whenever $|E|=|K|$.

In recent years, a lot of attention has been given to the stability of the isoperimetric/Wulff inequality.
A quantitative but not sharp form of general Wulff inequality is first given by \cite{EFT2005}.
Later, the  sharp stability of classical isoperimetric inequality was shown by Fusco, Maggi and Pratelli in \cite{FMP2008}, and later in \cite{FMP2010} this sharp result was extended to the general Wulff inequality. More precisely, in \cite{FMP2010} the authors proved that, for any
set of finite perimeter $E\subset \mathbb R^n$ with $|E|=|K|$, one has
\begin{equation}
\label{eq:stability}
\mathscr F(E) - \mathscr F(K)\ge c(n,K) \min_{y\in \mathbb R^n}\left\{ |E\Delta (y+K)|\right\}^2,
\end{equation}
where 
$$E\Delta F= (E\setminus F)\cup (F\setminus E)$$
denotes the symmetric difference between $E$ and $F$. Finally,  Fusco and Julin \cite{FJ2014} and Neumayer  \cite{N2016} generalized this inequality to a stronger form via a technique known as the selection principle, first introduced in this class of problems by Cicalese and Leonardi \cite{CL12}.
We recommend the survey papers \cite{F-EMS12,F2017} for more details. 
\\

In this paper we focus on the case when  $\mathscr F$ is crystalline, i.e.,\ there exists a finite set $\{x_j\}_{j=1}^M\subset \mathbb R^n\setminus\{0\},$ with $M\in \mathbb N$, such that
$$f(\nu)=\max_{1\le j\le M} x_j\cdot \nu, \ \ \nu \in\mathbb S^{n-1}.$$
Then the corresponding Wulff shape $K$ is a convex polyhedron, and the dual $f_*$ is of the form 
\begin{equation}
\label{eq:f*}
f_*(x)=\sup_{1\le i\le N} \sigma_i\cdot x,\ \ x\in \mathbb R^n,
\end{equation}
for some $N\in \mathbb N$.
Here $\sigma_i$ is a vector parallel to the normal $\nu_i$ of the face $\partial K\cap V_i$ of  $K$, where $V_i$ denotes the convex cone
$$V_i=\{x\in \mathbb R^n\colon f_*(x)=\sigma_i\cdot x\}.$$
We assume that the set of vectors $\sigma_i$ is ``minimal'', i.e. 
\begin{equation}
\label{eq:minimal}
\text{$|V_i\cap K|>0$ for any $i=1,\ldots,N$.}\footnote{Note that one could always artificially add some extra vectors $\sigma$ in the definition of $f_*$ by simply choosing $\sigma$ small enough so that $f_*(x)>\sigma\cdot x$ for any $x \neq 0$.
In this way, $f_*$ is unchanged by adding $\sigma$ to the set of vectors $\{\sigma_i\}_{1\leq i \leq N}$. Thus, asking that $|V_i\cap K|>0$ for any $i$ guarantees that all the vectors $\sigma_i$ play an active role in the definition of $f_*$.}
\end{equation}

Define $\mathscr C(K)$ as the collection of bounded open convex sets close to $K$, whose faces are parallel to those of $K$, and whose volume is $|K|$. To be specific, given $\a=(a_1,\,\cdots,\,a_N)\in \mathbb R^N$ with $|\a|< 1$, let
\begin{equation}\label{star}
f^\a_*(x)=\sup_{1\le i\le N} \frac{\sigma_i\cdot x}{1+a_i} \quad \text{ and } \quad    K^\a=\{f^\a_*(x)\le 1\} .
\end{equation}
Note that $f_*=f^0_*$. Then we define 
$$
\mathscr C(K)=\big\{K^\a\,:\,\,\a\in \mathbb R^N,\,|\a| < 1,\,|K^\a|=|K|\big\}.
$$

Following \cite{FM2011}, given a set $A$ and $R>0,$ we denote by $\mathcal N_{R,K}(A)$ the $R$-neighborhood of $A$ with respect to $K$,
namely
$$
\mathcal N_{R,K}(A)=\{x \in \R^n\,:\,{\rm dist}_K(x,A)\leq R\},\qquad \text{where}\quad {\rm dist}_K(x,A)=\inf_{y\in A}f_*(x-y).
$$
Then we say that a set $E$ is an $(\ez,R)$-minimizer for the surface energy $\mathscr F$ if, for every set of finite perimeter $G\subset \mathbb R^n$ satisfying  $|E|=|G|$ and $G\subset \mathcal N_{R,K}(E)$, one has
$$\mathscr F(E)\le  \mathscr F(G)+\ez \left(\frac {|K|}{|E|}\right)^{\frac 1 n} |E\Delta G|.$$
We shall also say that  $E$ is an $\ez$-minimizer  if  $E$ is an $(\ez,R)$-minimizer with $R=+\infty.$
It is clear from the definition that an  $(\ez_1, R_1)$-minimizer is an $(\ez_2, R_2)$-minimizer  whenever $\ez_1\le \ez_2$ and $R_2\ge R_1$.

The following result is the main theorem of the paper.

\begin{thm}\label{main thm}
	Let $\mathscr F$ be crystalline, $K$ be its Wulff shape, and $E$ be a set of finite perimeter with $|E|=|K|$. Then there exist constant $\ez_0=\ez_0(n,K)>0$ such that, for any $0<\ez\le \ez_0$ and $R \geq n+1$, if $E$ is an $(\ez,R)$-minimizer of the surface energy associated to $f$
	then, up to a translation, $E=K^\a$ for some  $K^\a\in \mathscr C(K)$.
\end{thm}
%

In \cite{FM2011} the following variational problem was considered
\begin{equation}
\label{eq:variational}
\min\left\{\mathscr F(E)+\int_{E} g\, dx \colon E\subset \mathbb R^n \text{ is a set of finite perimeter}\right\},
\end{equation}
where $g\colon \mathbb R^n\to [0,\,\infty)$ is a locally bounded Borel function with 
$g(x)\to \infty$
as $|x|\to \infty$.

In this model, $\mathscr F$ represents the surface energy of a droplet, while $g$ is a potential term. Hence, the minimization problem aims to understand the equilibrium shapes of droplets/crystals under the action of an external potential.
As shown in \cite[Corollary 2]{FM2011}, for $|E|\ll 1$, any minimizer of \eqref{eq:variational} is an $(\ez,n+1)$-minimizer of $\mathscr F$. Also, it was noted in \cite[Theorem 7]{FM2011}
that, when $n=2$, minimizers with sufficiently small mass are polyhedra with sides parallel to $K$.
It was then asked in \cite[Remark 1]{FM2011} whether this result would hold in every dimension.
Our results gives a positive answer to this question in all dimension, by simply applying Theorem \ref{main thm} to the rescaled set $\left(\frac{|K|}{|E|}\right)^{1/n}E$. We summarize this in the following:
\begin{cor}
\label{cor main}
Let $\mathscr F$ be crystalline, $K$ be its Wulff shape, and $g\colon \mathbb R^n\to [0,\,\infty)$ a locally bounded Borel function such that
$g(x)\to \infty$
as $|x|\to \infty$.
Let $E$ be a minimizer of \eqref{eq:variational}. There exists $m_0=m_0(n,K,g)>0$ small enough such that if $|E|\leq m_0$ then $\left(\frac{|K|}{|E|}\right)^{1/n}E \in \mathscr C(K).$ 
\end{cor}
This result is particularly interesting from a ``numerical'' viewpoint: since the space $\mathscr C(K)$ is finite dimensional, minimizers of \eqref{eq:variational} can be found explicitly by minimizing the energy functional over a small neighborhood of $K$ in this finite dimensional space.
In addition, in some explicit cases, it can be used to analytically find the exact minimizer.
For instance, in \cite{CNT19}, the authors used the two-dimensional analogue of Corollary 
\ref{cor main} to explicitly find the minimizers of some variational problems coming from plasma physics. Thanks to Corollary \ref{cor main} one can now perform the same kind of analysis also in the physical dimension $n=3$.\\

We note that while Theorem \ref{main thm} proves that all $(\ez,R)$-minimizers with $R\geq n+1$  are in $\mathscr C(K)$,
one may wonder whether all elements of $\mathscr C(K)$ are $(\ez,R)$-minimizers for $R\geq n+1$. 
This is our second result, which gives a full characterization of $\ez$-minimizers for $\mathscr F$ when $\ez$ and $|\a|$ are sufficiently small.

\begin{thm}\label{characterization}
	Let $\mathscr F$ be crystalline, $K$ be its Wulff shape,
	and $\ez_0$ be as in Theorem \ref{main thm}. Then, for any $\ez\leq \ez_0$ there exists $a_0=a_0(n,\,K,\,\ez)>0$ such that, for any $0<|\a|\le a_0$, any set $K^\a$ is an $\ez$-minimizer (and so also an $(\ez,R)$-minimizer for all $R\geq n+1$)  for $\mathscr F$.
	\end{thm}

Finally, with the help of Theorem~\ref{main thm} and ~\ref{characterization} we show the following stability inequality.

\begin{thm}\label{stability}
Let $\mathscr F$ be crystalline, and let $K$ be its Wulff shape. Then there exists $\sigma_0=\sigma_0(n,K)>0$ such that, for any set of finite perimeter $E$ satisfying $|E\Delta K| \leq \sigma_0$, there exists $K^\a\in \mathscr C(K)$ such that
$$
\mathscr F(E) - \mathscr F(K^\a)\ge \gamma |E \Delta K^\a|.
$$
\end{thm}

The set $K^\a$ is given via Lemma~\ref{Ka of E} below, and can be informally thought as a sort of ``projection'' of $E$ on the set $\mathscr C(K).$
An immediate corollary of Theorem~\ref{stability} is that, under the constraint
$$|E\cap V_i^\a|=|K^\a\cap V_i^\a|\qquad \forall\,i=1,\ldots,N,$$
$K^\a$ is the unique minimizer for $\mathscr F$, where
$$V_i^\a=\left\{x\in \mathbb R^n\colon f^\a_*(x)= \frac{\sigma_i\cdot x}{1+a_i}\right\}.$$ 

The stability inequality provided by Theorem \ref{stability} is different from the classical one in \eqref{eq:stability}. Indeed, first of all elements in $\mathscr C(K)$ might not be a minimizer of  $ \mathscr F$ (only $K$ and its translates are minimizers). In addition, thanks to the non-smoothness of $f$, we are able to obtain a stability result for the isoperimetric inequality with a linear control (power $1$) on $|E \Delta K^\a|$.

One could understand this result geometrically in the following way:\\ If we plot the ``graph'' of $\mathscr F$ in a neighborhood of $K$, $\mathscr F$ is constant on the translates of $K$, while it increases smoothly on the space $\mathscr C(K)$ near $K$; see e.g.\ Lemma~\ref{lem:F C1} below. On the other hand, given a set $E$ close to $K$ but not in $\mathscr C(K)$,  the value of $\mathscr F$ grows very fast (linearly) from the value of $\mathscr F$ on the ``best approximation of $E$ in $\mathscr C(K)$'' provided by Lemma \ref{Ka of E}.
In other words, the energy $\mathscr F$ varies smoothly on $\mathscr C(K),$ while it has a Lipschitz singularity when moving away from $\mathscr C(K).$
\\

The paper is organized as follows. In Section 2, we introduce a method to associate any set of finite perimeter $E$ with small $|E\Delta K|$ and $|E|=|K|$, and we study the behaviour of $\mathscr F$ on $\mathscr C(K)$. Then we prove Theorem~\ref{main thm} in Section 3. Finally, the proofs of Theorem~\ref{characterization} and Theorem~\ref{stability} are given in the last section. Some useful technical results are collected in the appendix.

\section{Associate $K^\a$ to a set of finite perimeter $E$}
We begin by setting the notation used in this paper.
We denote by $C=C(\cdot)$ and $c=c(\cdot)$ positive constant, with
the parentheses including all the parameters on which the constants depend. The constants $C(\cdot)$ and $c(\cdot)$ may
vary between appearances, even within a chain of inequalities. 
We denote by $|A|$ the Lebesgue measure of $A$, and by $\mathcal H^{\az}(A)$ the $\az$-dimensional Hausdorff measure of $A$. 
The reduced boundary of a set of finite perimeter $E$ is denoted by $\partial^* E$. By $\dist(A,\,B)$ we denote the Euclidean distance between $A$ and $B$.

As in the introduction, $K$ is the Wulff shape associated to $f_*$. We denote by $rK$ the dilation of $K$ by a factor $r$ with respect to the origin. 
Also, given $\a \in \R^N$ with $|\a|<1$, we define
$$V_i^\a= \left\{x\in \mathbb R^n\colon f^\a_*(x)= \frac{\sigma_i\cdot x}{1+a_i}\right\},$$
where $\{\sigma_i\}_{1\leq i\leq N}$ are the vectors defining $f_*$ (see \eqref{eq:f*}).

For a vector $x\in\mathbb{R}^n$, we denote by $x'$ its first $(n-1)$-coordinates, and by $x_n$ its last coordinate. We denote by $\nu_E$ is the measure-theoretic outer unit normal to a set of finite perimeter $E$, and omit the subindex if the set $E$ in question is clear from the context.

We first recall the following result. 
\begin{lem}[{\cite[Lemma 5 and Theorem 5]{FM2011}}]
\label{lem:E close K}
If E is an $(\ez,n+1)$-minimizer of $\mathscr F$ with $|E|=|K|$, then there exists $x\in \mathbb R^n$ so that
$$ |(x+E)\Delta K|\le C(n)|K|\ez. $$
In addition, up to a translation, $\partial E$ is uniformly close to $\partial K$, namely
$$
(1-r)K\subset  x+E\subset (1+r)K,\qquad r \leq C(n,K)\ez^{1/n}.
$$
\end{lem}
The lemma above will allows us to apply the following result to $(\ez,n+1)$-minimizers, when $\ez$ is small.

\begin{lem}\label{Ka of E}
Let $E\subset \mathbb R^n$ be a set of finite perimeter and $|E|=|K|$. There exist $\eta=\eta(n,K)>0$ and $C=C(n,K)>0$ such that the following holds:
If 
$$
(1-\eta)K\subset  E\subset (1+\eta)K
$$
then there exists a unique $\a=\a(K,\,|E\Delta K|)\in \mathbb R^N$, with $|\a|\leq C|E\Delta K|$, such that
$$|E\cap V_i^\a|=|K^\a\cap V_i^\a|\qquad \forall\,i=1,\ldots,N,$$
for some $K^\a \in \mathscr C(K)$.
\end{lem}
\begin{proof}
Note that, by the definition of $V_i$ and $K$,
$$
\partial K\cap V_i\subset \{\sigma_i \cdot x=1\}\qquad \forall\, i=1,\ldots,N.
$$
Given $t>0$, set $H_i^t:=\{x\in \R^n \,:\,\sigma_i\cdot x<t\}$.
Then it follows by our assumption on $E$ that, for any $\a$ small enough,
\begin{equation}
\label{eq:E Via}
E\cap V_i^\a\supset H_i^{1-2\eta}\cap V_i^\a\qquad \forall\, i=1,\ldots,N.
\end{equation}
Similarly, for $\a$ small, 
\begin{equation}
\label{eq:K Via}
K^\a\cap V_i^\a\supset H_i^{1-2\eta}\cap V_i^\a\qquad \forall\, i=1,\ldots,N.
\end{equation}
Consider the map
$$
\phi\colon \a\mapsto (|K^\a\cap V_1^\a|-|E\cap V_1^\a|,\ldots, |K^\a\cap V_N^\a|-|E\cap V_N^\a|)=(\phi_1(\a),\ldots,\phi_N(\a)).
$$
Our goal is to show that $\phi(\a)=0$ for a unique  vector $\a$ satisfying $|\a|\leq C|E\Delta K|$.

It is easy to check that $\phi$ is Lipschitz, and we now want to compute its differential.
Note that, thanks to \eqref{eq:E Via} and \eqref{eq:K Via},
$$
\phi_i(\a)=|(K^\a\cap V_i^\a)\setminus H_i^{1-2\eta}|-|(E\cap V_i^\a)\setminus H_i^{1-2\eta}|=:\phi_{i,1}(\a)+\phi_{i,2}(\a).
$$
Let $d_i^\a$ be the distance from the origin to  $\partial K^\a\cap V_i^\a$. Then, by \eqref{star} we have
$$d_i^\a=\frac{1+a_i}{|\sigma_i|},$$
hence
\begin{align*}
|(K^\a\cap V_i^\a)\setminus H_i^{1-2\eta}|&=|K^\a\cap V_i^\a|-|H_i^{1-2\eta}|\cap V_i^\a|\\
&=\frac 1 n\mathcal H^{n-1}(\partial K^\a \cap V_i^\a)\bigl(d_i^\a-(1-2\eta)d_i\bigr)\\
&=\frac 1 n\mathcal H^{n-1}(\partial K^\a \cap V_i^\a)\frac{2\eta+a_i}{|\sigma_i|},
\end{align*}
and therefore
\begin{multline*}
\phi_{i,1}(\a')-\phi_{i,1}(\a)-\frac 1 n\frac{\mathcal H^{n-1}(\partial K^\a \cap V_i^\a)}{|\sigma_i|}(a_i'-a_i)\\
=\frac{2\eta+a_i'}{n|\sigma_i|}  \Bigl(\mathcal H^{n-1}(\partial K^{\a'} \cap V_i^{\a'})-\mathcal H^{n-1}(\partial K^\a \cap V_i^\a)\Bigr).
\end{multline*}
Thus, defining $$v_i:=|K \cap V_i |=\frac1n \frac{\mathcal H^{n-1}(\partial K \cap V_i)}{|\sigma_i|}, $$
since the map $\a \mapsto \mathcal H^{n-1}(\partial K^\a \cap V_i^\a)$ is Lipschitz (although not needed here, in Lemma \ref{lem:F C1} we also show how compute its differential) it follows that
$$
|\phi_{i,1}(\a')-\phi_{i,1}(\a)-v_i(a_i'-a_i)|\leq C(n,K)\,\bigl(\eta+|\a'|+|\a|\bigr)|\a'-\a|.
$$
On the other hand, since $E\subset (1+\eta)K$ and $\cap_{i=1}^N H_i^{1-2\eta}=(1-2\eta)K$, we can bound
$$
|\phi_{i,2}(\a')-\phi_{i,2}(\a)|\leq \bigl|(V_i^{\a'}\Delta V_i^\a) \cap \bigl((1+\eta)K\setminus (1-2\eta)K\bigr)\bigr| \leq C(n,K)\,\eta\,|\a'-\a|.
$$
Thus, if we consider the linear map $A:\R^N\to \R^N$ given by
$$A:=\left( \begin{array}{cccc}
v_1 & 0& \cdots  &0 \\
0 & v_{2} & \ddots &\vdots \\
\vdots  & \ddots & \ddots& 0 \\
0 &\cdots &0 & v_N
\end{array} \right),
$$
we proved that
$$
|\phi(\a')-\phi(\a)-A(\a'-\a)|\leq C(n,K)\,\bigl(\eta+|\a'|+|\a|\bigr)|\a'-\a|.
$$
In particular, at every differentiability point of $\phi$ we have
\begin{equation}
\label{eq:Taylor2}
|D\phi(\a)-A|\leq C(n,K)\,\bigl(\eta+|\a|\bigr).
\end{equation}
Since $\phi$ is differentiable a.e. (by Rademacher Theorem),
for every $\a$ we can define $\partial_C\phi(\a)$ as the convex hull of all the limit of gradients, namely
$$
\partial_C\phi(\a):={\rm co}\Bigl\{\lim_{j\to \infty}D\phi(a_j)\,:\,\a_j\to \a,\,\phi \text{ is differentiable at }\a_j \Bigr\},
$$ 
and it follows by \eqref{eq:Taylor2} that 
\begin{equation}
\label{eq:Taylor2}
|A_\a-A|\leq C(n,K)\,\bigl(\eta+|\a|\bigr) \qquad \forall\,A_\a \in \partial_C\phi(\a).
\end{equation}
Since $A$ has rank $N$ (recall that $v_i>0$ for any $i$, see \eqref{eq:minimal}), there exists $r_0>0$ such that $A_\a$ is invertible for every $A_\a \in \partial_C\phi(\a)$,
 provided $\eta\leq r_0$ and $|\a|\leq r_0$.
This allows us to apply the inverse function theorem for Lipschitz mappings \cite{Cla76} to deduce that
$$
\phi:B_{r_0}(0)\to \phi(B_{r_0}(0))
$$
is a bi-Lipschits homeomorphism, with bi-Lipschitz constant depending only on $n$ and $K$. In particular, since $|\phi(0)| \leq C(n,K)|E\Delta K| \leq C(n,K)\eta$ (by our assumption on $E$),
if $\eta=\eta(n,K)$ is sufficiently small with respect to $r_0$, then $\phi(B_{r_0}(0))$ contains the origin.

This means exactly that there exists a unique $\a \in B_{r_0}(0)$ such that $\phi(\a)=0$.
In addition, again by bi-Lipschitz regularity of $\phi$,
$$
|\phi(\a)-\phi(0)|=|\phi(0)|\leq C(n,K)|E\Delta K|
$$
implies that $|\a|\leq C(n,K)|E\Delta K|$, as desired.
\end{proof}
We conclude this section with a simple result on the behaviour of $\mathscr F$ on $\mathscr C(K)$,
showing that $\mathscr F|_{\mathscr C(K)}$ is $C^1$ at $K$, with zero gradient.

\begin{lem}
\label{lem:F C1}
There exists a modulus of continuity $\omega:\R^+\to \R^+$ such that the following holds: if
$|\a|+|\a'|<1$
 and $|K^\a|=|K^{\a'}|$, then
$$
\bigl|\mathscr F(K^\a)-\mathscr F(K^{\a'})\bigr| \leq \omega(|\a|+|\a'|)|\a-\a'|.
$$
\end{lem}
\begin{figure}[ht]
	\centering
	\def\svgwidth{300 pt}
\begingroup%
  \makeatletter%
  \providecommand\color[2][]{%
    \errmessage{(Inkscape) Color is used for the text in Inkscape, but the package 'color.sty' is not loaded}%
    \renewcommand\color[2][]{}%
  }%
  \providecommand\transparent[1]{%
    \errmessage{(Inkscape) Transparency is used (non-zero) for the text in Inkscape, but the package 'transparent.sty' is not loaded}%
    \renewcommand\transparent[1]{}%
  }%
  \providecommand\rotatebox[2]{#2}%
  \newcommand*\fsize{\dimexpr\f@size pt\relax}%
  \newcommand*\lineheight[1]{\fontsize{\fsize}{#1\fsize}\selectfont}%
  \ifx\svgwidth\undefined%
    \setlength{\unitlength}{209.76377953bp}%
    \ifx\svgscale\undefined%
      \relax%
    \else%
      \setlength{\unitlength}{\unitlength * \real{\svgscale}}%
    \fi%
  \else%
    \setlength{\unitlength}{\svgwidth}%
  \fi%
  \global\let\svgwidth\undefined%
  \global\let\svgscale\undefined%
  \makeatother%
  \begin{picture}(1,0.40540541)%
    \lineheight{1}%
    \setlength\tabcolsep{0pt}%
    \put(0,0){\includegraphics[width=\unitlength,page=1]{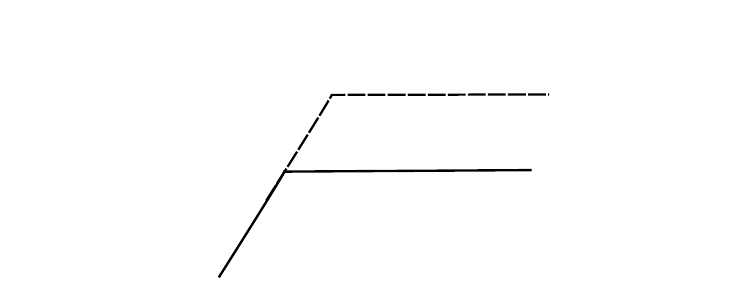}}%
    \put(0.40777825,0.17696494){\color[rgb]{0,0,0}\makebox(0,0)[lt]{\lineheight{1.25}\smash{\begin{tabular}[t]{l}$\theta_{ij}$\end{tabular}}}}%
    \put(0,0){\includegraphics[width=\unitlength,page=2]{surface.pdf}}%
    \put(0.69585687,0.18455825){\color[rgb]{0,0,0}\makebox(0,0)[lt]{\lineheight{1.25}\smash{\begin{tabular}[t]{l}$K^{\a}$\end{tabular}}}}%
    \put(0.69060259,0.29173872){\color[rgb]{0,0,0}\makebox(0,0)[lt]{\lineheight{1.25}\smash{\begin{tabular}[t]{l}$K^{\a'}$\end{tabular}}}}%
    \put(0,0){\includegraphics[width=\unitlength,page=3]{surface.pdf}}%
    \put(0.2529169,0.2917416){\color[rgb]{0,0,0}\makebox(0,0)[lt]{\lineheight{1.25}\smash{\begin{tabular}[t]{l}$b^\a_{ij}(d_i^{\a'}-d_i^\a)d_j \sin^{-1}(\theta_{ij})$\end{tabular}}}}%
    \put(0,0){\includegraphics[width=\unitlength,page=4]{surface.pdf}}%
    \put(0.4284701,0.08010112){\color[rgb]{0,0,0}\makebox(0,0)[lt]{\lineheight{1.25}\smash{\begin{tabular}[t]{l}$b^\a_{ij}(d_i^{\a'}-d_i^\a)d_i \tan^{-1}(\theta_{ij})$\end{tabular}}}}%
  \end{picture}%
\endgroup%

	\caption{We show how moving the faces   changes of the surface tension.}
	\label{fig:surface}
\end{figure}

\begin{proof}
Obviously it suffices to prove the result when $|\a'-\a|$ is small.
In this proof we denote by $o(1)$ a quantity that goes to $0$ as $|\a|+|\a'|\to 0$.

By writing $b^\a_{ij}=\mathcal H^{n-2}(\partial K^\a\cap \partial V_i^\a\cap \partial V_j^\a)$ (note that $b^\a_{ij}=0$ when the two sets $V_i^\a$ and $V_j^\a$ are not adjacent) and choosing $\theta_{ij} \in (0,\pi)$ so that
$\cos{\theta_{ij}}=\nu_i\cdot \nu_j$, a simple geometric construction (see Figure~\ref{fig:surface}) shows that, at first order in $\a'-\a$, the surface energy of $K^{\a'}$ inside $V_i^\a$ has an extra term 
$b^\a_{ij}(d_i^{\a'}-d_i^\a)f(\nu_j) \sin^{-1}(\theta_{ij})$ with respect to the surface 
energy of $K^{\a}$,
but also a negative term $-b^\a_{ij}(d_i^{\a'}-d_i^\a)f(\nu_i) \tan^{-1}(\theta_{ij})$.

Since $f(\nu_i)=d_i$ and $f(\nu_j)=d_j$ (this follows by the relation between $f$ and $K$), this gives
\begin{equation}
\label{eq:Faa'}
\begin{split}
&\mathscr F(K^{\a'})-\mathscr F(K^\a)\\
=&\sum_{1\le i ,\, j\le N,\,i\neq j} b^\a_{ij}(d_i^{\a'}-d_i^\a)d_j \sin^{-1}(\theta_{ij})-b^\a_{ij}(d_i^{\a'}-d_i^\a)d_i \tan^{-1}(\theta_{ij})+o(1)|\a-\a'|\\
=&\sum_{1\le i,\,j\le N,\,i\neq j}b^\a_{ij}(d_i^{\a'}-d_i^\a)[d_j \sin^{-1}(\theta_{ij})-d_i \tan^{-1}(\theta_{ij})]+o(1)|\a-\a'|.
\end{split}
\end{equation}
Analogously, 
the difference of the volumes of $K^{\a'}$ and $K^\a$ inside $V_i^\a$ is given, at first order in $\a'-\a$, by
$(d_i^{\a'}-d_i^\a)\mathcal H^{n-1}(\partial K^\a\cap V_i^\a)$,
therefore
$$
|K^{\a'}|-|K^\a|=\sum_{1\le i\le N}(d_i^{\a'}-d_i^\a)\mathcal H^{n-1}(\partial K^\a\cap V_i^\a)+o(1)|\a-\a'|.$$
Let us now denote by  $\pi_i^\a(O)$ the orthogonal projection of the origin onto the hyperplane $H_i^\a$ containing $\partial K^\a\cap V_i^\a$, let $H_{ij}^\a\subset H_i^\a$ denote the half-hyperplane such that 
$$
\partial H_{ij}^\a\supset \partial V_i^\a\cap \partial V_j^\a \qquad \text{and}\qquad H_{ij}^\a\supset \partial K^\a\cap V_i^\a,
$$
and let us define
$\sdist(\pi_i^\a(O),\,\partial H^\a_{ij})$ the ``signed distance of $\pi_i^\a(O)$ from $\partial H_{ij}^\a$'', namely 
$$
\sdist(\pi_i^\a(O),\,\partial H^\a_{ij})
=\left\{
\begin{array}{ll}
\dist(\pi_i^\a(O),\,\partial H_{ij}^\a) &\text{if }\pi_i^\a(O) \in H^\a_{ij},\\
-\dist(\pi_i^\a(O),\,\partial H_{ij}^\a) &\text{if }\pi_i^\a(O) \not \in  H_{ij}^\a.
\end{array}
\right.
$$
Then, with this definition, using  the formula for the volume of a cone we get
$$
\mathcal H^{n-1}(\partial K^\a\cap V_i^\a)= \frac1n\sum_{j \neq i} b^\a_{ij} \sdist(\pi_i^\a(O),\,\partial H^\a_{ij}),
$$
from which we deduce that
$$
|K^{\a'}|-|K^\a|=\frac 1 n\sum_{1\le i,\,j\le N,\,i\neq j}b^\a_{ij}(d_i^{\a'}-d_i^\a)\sdist(\pi_i^\a(O),\,\partial H^\a_{ij})+o(1)|\a-\a'|.$$
Since by assumption $|K^{\a'}|=|K^\a|$, this proves that 
\begin{equation}
\label{eq:vol aa'}
\sum_{1\le i,\,j\le N,\,i\neq j}b^\a_{ij}(d_i^{\a'}-d_i^\a)\sdist(\pi_i^\a(O),\,\partial H^\a_{ij})=o(1)|\a-\a'|.
\end{equation}

\begin{figure}[ht]
	\centering
	\def\svgwidth{300 pt}
	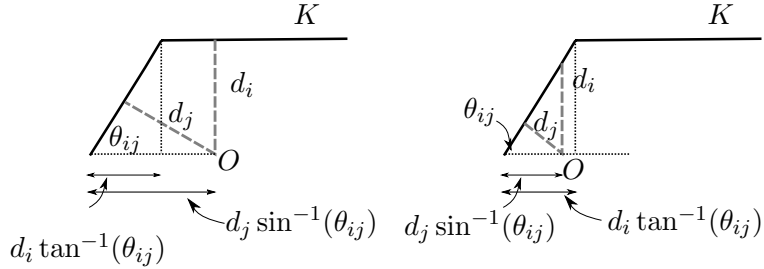

	\caption{We illustrate why formula \eqref{geo} holds. On the left-hand side, we consider the case when the origin $O$ projects inside the $i$-th fact of $\partial K$, while on the right-hand side  the case when the origin $O$ projects outside the $i$-th fact of $\partial K$ (so in this case $\sdist(\pi(O),\partial H_{ij})$ is negative).  }
	\label{fig:surface2}
\end{figure}

Using $\pi_i$ (resp. $\partial H_{ij}$) to denote $\pi_i^0$ (resp. $\partial H^0_{ij}$), we see that 
$$
\sdist(\pi_i^\a(O),\,\partial H^\a_{ij})=\sdist(\pi_i(O),\,\partial H_{ij})+o(1).
$$
Therefore, since $|d_i^{\a'}-d_i^\a|\leq C|\a'-\a|$,
it follows by \eqref{eq:vol aa'} that
\begin{equation}
\label{eq:vol aa'2}
\sum_{1\le i,\,j\le N,\,i\neq j}b^\a_{ij}(d_i^{\a'}-d_i^\a)\sdist(\pi_i(O),\,\partial H_{ij})=o(1)|\a-\a'|.
\end{equation}
Noticing now that
\begin{equation}
\label{geo}
d_j \sin^{-1}(\theta_{ij})-d_i \tan^{-1}(\theta_{ij})=\sdist(\pi_i(O),\,\partial H_{ij}), 
\end{equation}
(see Figure~\ref{fig:surface2}),
combining \eqref{eq:Faa'}, \eqref{eq:vol aa'2}, and \eqref{geo} we obtain that 
$$\mathscr F(K^{\a'})-\mathscr F(K^\a)=o(1)|\a-\a'|,$$ 
as desired.
\end{proof}


\section{Proof of Theorem~\ref{main thm}}

Let $E\subset \mathbb R^n$ be an $(\ez,R)$-minimizer with $|E|=|K|$ and $R\geq n+1$. Thanks to Lemma \ref{lem:E close K}, we know that for $\ez$ small enough $E$ is close to a translation of $K$. In particular, up to a translation, we can apply Lemma \ref{Ka of E} to $E$. We denote by $K^\a\in \mathscr C(K)$ the set provided by such lemma. 

\subsection{An almost identity map}

\begin{figure}[ht]
	\centering
	\def\svgwidth{300 pt}
\begingroup%
  \makeatletter%
  \providecommand\color[2][]{%
    \errmessage{(Inkscape) Color is used for the text in Inkscape, but the package 'color.sty' is not loaded}%
    \renewcommand\color[2][]{}%
  }%
  \providecommand\transparent[1]{%
    \errmessage{(Inkscape) Transparency is used (non-zero) for the text in Inkscape, but the package 'transparent.sty' is not loaded}%
    \renewcommand\transparent[1]{}%
  }%
  \providecommand\rotatebox[2]{#2}%
  \newcommand*\fsize{\dimexpr\f@size pt\relax}%
  \newcommand*\lineheight[1]{\fontsize{\fsize}{#1\fsize}\selectfont}%
  \ifx\svgwidth\undefined%
    \setlength{\unitlength}{209.76377953bp}%
    \ifx\svgscale\undefined%
      \relax%
    \else%
      \setlength{\unitlength}{\unitlength * \real{\svgscale}}%
    \fi%
  \else%
    \setlength{\unitlength}{\svgwidth}%
  \fi%
  \global\let\svgwidth\undefined%
  \global\let\svgscale\undefined%
  \makeatother%
  \begin{picture}(1,0.40540541)%
    \lineheight{1}%
    \setlength\tabcolsep{0pt}%
    \put(0,0){\includegraphics[width=\unitlength,page=1]{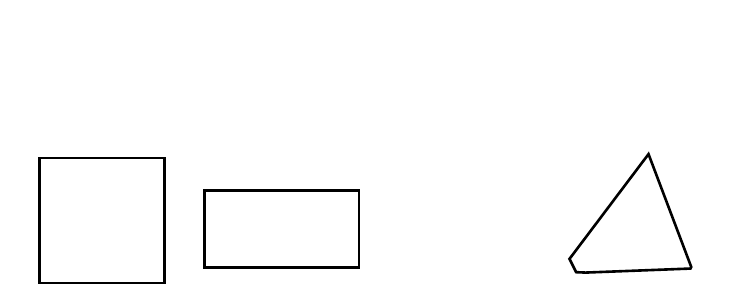}}%
    \put(0.31183516,0.24414093){\color[rgb]{0,0,0}\makebox(0,0)[lt]{\lineheight{1.25}\smash{\begin{tabular}[t]{l}$K$\end{tabular}}}}%
    \put(0,0){\includegraphics[width=\unitlength,page=2]{diff.pdf}}%
    \put(0.62947433,0.23870164){\color[rgb]{0,0,0}\makebox(0,0)[lt]{\lineheight{1.25}\smash{\begin{tabular}[t]{l}$K^\a$\end{tabular}}}}%
    \put(0,0){\includegraphics[width=\unitlength,page=3]{diff.pdf}}%
  \end{picture}%
\endgroup%

	\caption{The sets $K$ and $K^\a$ do not necessarily have the same number of vertices. For example, when $K$ is a pyramid and $K^\a$ is obtained by moving one of the lateral sides of $K$, the number of vertices may increase.\newline
If one looks at one face of $\partial K$ and the corresponding one in $\partial K^\a$, whenever the number of vertices of that face did not increase (as in the square depicted above) then one can find a nice transformation between the two faces. Instead, whenever the number of vertices increases (as in the case of the triangle), then there is a well defined map from the face of $\partial K^\a$ to the one of $\partial K$ which collapses the small triangle onto a segment.
}
	\label{fig:diff}
\end{figure}

Recall that $\partial K^\a\cap V_i^\a$ is parallel to $\partial K\cap V_i$ for all $1\le i\le N$. Moreover by the construction of $K^\a$, there is a natural map between the vertices of  $\partial K^\a$ to those of $\partial  K$ when $|\a|$ is sufficiently small: this map sends each vertex of $\partial K^\a$ to the closest one on $\partial K$. We note that this map is not necessarily one to one, see Figure~\ref{fig:diff}. 

Starting from this map, we construct a map $\phi\colon \partial K^\a\to \partial  K$ as follows: Given a face of $\partial K^\a$,
we split it into 
 a union of simplices with basis at the barycenter of the face, and then we map  each symplex onto the corresponding one of $\partial K$ using an affine map that maps vertices to vertices and the barycenter to the barycenter.
This produces a global piecewise-affine map $\phi$ from $\partial K^\a$ onto $\partial K$.

Note that, if the number of vertices of  a face of $K^\a$ is same as the one of $K$, since $K^\a$ is a small perturbation of $K$, then the corresponding matrix for each of the affine maps is of the following form
${\rm Id}_{n-1}+o(1)A$,
where ${\rm Id}_{n-1}$ is the identity matrix in $\R^{(n-1)\times (n-1)}$, $o(1)$ is a quantity going to zero as $|\a|\to 0$, and $A \in \R^{(n-1)\times (n-1)}$ is a matrix with  bounded entries. 

On the other hand, if some vertices of $\partial K^\a$ are mapped to the same one, 
then the map will be degenerate in some direction. In particular, in some suitable system of coordinates, it will be of the form ${\rm Id}_m+o(1)A$,
where  ${\rm Id}_{m}$ is the identity matrix in $\R^{m}$ for some $m\leq n-2$, $o(1)$ is a quantity going to zero as $|\a|\to 0$, and $A \in \R^{(n-1)\times (n-1)}$ is a matrix with  bounded entries. 

In this way, we obtain a map
$\phi\colon\partial K^\a \to \partial K$ so that, for each $i=1,\ldots,N,$
\begin{equation}\label{dist compare}
\phi(\partial K^\a\cap V_i^\a)= \partial K\cap V_i,\qquad
\dist(x,\,\partial V^\a_i)\le C(n,K)\,\dist(\phi(x),\,\partial V_i)\quad \forall\,x \in \partial K^\a\cap V_i^\a.
\end{equation}
Also, the following a control on its tangential divergence holds:
\begin{equation}\label{almost identity}
{\rm div}_{\tau}\phi\le n-1+o(1)\qquad \text{on }\partial K^\a,
\end{equation}
and
\begin{equation}\label{almost identity 1}
{\rm div}_{\tau}\phi= n-1+o(1) \qquad \text{on }\partial K^\a\setminus T,
\end{equation}
where $\mathcal H^{n-1}(T)=o(1)$ ($T\subset \partial K^\a$ is given by the union of the simplices on which the map $\phi$ is degenerate).

We now consider the vector field $X:\R^n\setminus \{0\}\to \partial K$ defined as
$$X(x):= \phi\left(\frac {x}{f_*^\a(x)}\right). $$
Since $f_*^\a\equiv 1$ on $\partial K^\a$, it follows by \eqref{almost identity} and \eqref{almost identity 1} that
\begin{equation}\label{almost identity 2}
 {\rm div}X\le \frac{n-1+o(1)}{f_*^\a(x)}\qquad \text{on }\R^n,
\end{equation}
and
\begin{equation}\label{almost identity 3}
 {\rm div}X= \frac{n-1+o(1)}{f_*^\a(x)} \qquad \text{on }\R^n \setminus \mathcal{C}_T,
\end{equation}
where $\mathcal{C}_T$ denotes the cone generated by $T$, namely $\mathcal{C}_T=\{tx\colon t>0,\,x \in T\}$.
 
Also, since $X(x)\in K$, we have 
$$f(\nu)=\sup_{z\in K} \nu\cdot z\ge \nu\cdot X(x) \qquad \forall\, x \in \R^n\setminus \{0\}.$$
Furthermore, we note that $X(x)\in \partial K\cap V_i$ for each $x\in V_i^\a\setminus \{0\}$. In particular, 
$f(\nu)=X\cdot \nu$
on $\partial K^\a$.
Thanks to this, we get
\begin{align}
\mathscr F(E)-\mathscr F(K^\a)&=\int_{\partial^* E} f(\nu)\,d\mathcal H^{n-1}-\int_{\partial^* K^\a} f(\nu)\,d\mathcal H^{n-1} \nonumber\\
&=\int_{\partial^* E} f(\nu)\,d\mathcal H^{n-1}-\int_{\partial^* K^\a} X\cdot \nu\,d\mathcal H^{n-1}\nonumber \\
&=\int_{\partial^* E} [f(\nu)-X\cdot \nu]\,d\mathcal H^{n-1}+ \int_{\partial^* E} X\cdot \nu\,d\mathcal H^{n-1}-\int_{\partial^* K^\a} X\cdot \nu\,d\mathcal H^{n-1} \label{A+B}
\end{align}
Set
$$I:=\int_{\partial^* E} [f(\nu)-X\cdot \nu]\,d\mathcal H^{n-1}$$
and
$$II:=\int_{\partial^* E} X\cdot \nu\,d\mathcal H^{n-1}-\int_{\partial^* K^\a} X\cdot \nu\,d\mathcal H^{n-1}.$$
We estimate $I$ and $II$ in the following two subsections, respectively.

\subsection{Estimate on I}

In this section we prove that, if $|\a|$ is sufficiently small, then
\begin{equation}\label{main aim 1}
I\ge c(n,K)\,|E\Delta K^\a|
\end{equation}

By additivity of the integral, it suffices to prove that 
$$
\int_{\partial^* E\cap V_i^\a} [f(\nu)-X\cdot \nu]\,d\mathcal H^{n-1}\ge c(n,K)\, |(E\Delta K^\a)\cap V_i^\a|\qquad \forall\,i=1,\ldots,N.
$$
To simplify the statement, we remove the subscript $i$. Thus, we assume that $V^\a$ is one of the cones for $K^\a$, and $V$ the corresponding cone for $K$. Also, up to a change of coordinate we can assume that for $x\in V^\a$ we have $f^\a_*(x)=\az x_n$ for some $\az>0$ (equivalently, $\partial K^\a\cap V^\a$ is contained in the hyperplane $\{x_n=\az^{-1}\}$). 
Then our aim is to show 
\begin{equation}
\label{eq:goal I}
\int_{\partial^* E\cap V^\a} [f(\nu)-X\cdot \nu]\,d\mathcal H^{n-1}\ge c(n,K)\, |(E\Delta K^\a)\cap V^\a|.
\end{equation}
Note that, since $\partial^* E$ is close to $\partial K$ (see Lemma \ref{lem:E close K}), $\partial^* E$ is uniformly away from the origin and $X$ is well-defined and uniformly Lipschitz on $\partial^* E$.

By the definition of $f$ and \eqref{dist compare}, for any outer normal $\nu$ at $x\in \partial^* E\cap V$  we have
$$ f(\nu) -\nu\cdot X\ge \sup_{y\in \partial K\cap \partial V} \nu\cdot y -\nu\cdot X.$$
Since $X(x) \in \partial K\cap V$ for $x \in V^\a\setminus \{0\}$, it follows that
$$
\text{$y-X(x)$ is parallel to $\partial K\cap V\quad$ for all $x \in \partial^*E\cap V,\,y \in \partial K\cap \partial V$.}
$$ 
In particular, denoting by $\nu'$ the projection of $\nu$  onto the first $(n-1)$-variables in $\mathbb R^{n}$ (namely, $\nu=(\nu',\nu_n)\in \R^{n-1}\times \R$), choosing $y\in \partial K\cap \partial V$ so that $y-X$ is parallel to $\nu'$ and recalling \eqref{dist compare}, we deduce that
\begin{multline*}
f(\nu) -\nu\cdot X\ge \nu\cdot (y -  X)\ge c(n,K) \,|\nu'| \dist(X(x),\,\partial K\cap \partial V)\\
\ge c(n,K)\, |\nu'| \dist\Big(\frac{x}{f_*^\a(x)},\, \partial V^\a\Big) \ge c(n,K)\, |\nu'| \dist(x,\, \partial V^\a),
\end{multline*}
where the last inequality follows from the fact that $f_*^\a\leq 2$ on $\partial^*E$ (since $\partial^* E$ is close to $\partial K$).

 Then, by the coarea formula on rectifiable sets (see e.g. \cite[Theorem 18.8]{M2012}), 
 \begin{equation}
 \label{eq:coarea}
 \begin{split}
 \int_{\partial^* E\cap V^\a} [f(\nu)-X\cdot \nu]\,d\mathcal H^{n-1} 
&\ge  c(n,K)  \int_{\partial^*  E\cap V^\a } |\nu'| \dist(x,\,\partial V^\a) \, d\mathcal H^{n-1}\\
&=  c(n,K)  \int_{\mathbb R} \int_{  (\partial^*  E\cap V^\a)_t } \dist(x,\,\partial V^\a)  \, d\mathcal H^{n-2} \, dt,
 \end{split}
 \end{equation}
where, given a set $F$, we denote by $(F)_t$ the slice at height $t$, that is
$$(F)_t:=\{x\in F \colon x_n=t\}. $$
Recalling that $\partial K^\a\cap V^\a\subset \{y\in \mathbb R^n\colon y_n=\az^{-1} \}$, we now consider two cases, depending on whether
$$\mathcal H^{n-1}((E)_{\az^{-1}} \cap V^\a )\le \frac 1 2 \mathcal H^{n-1}((V^\a)_{\az^{-1}})$$
or not.

In the first case,
since $\partial  E\cap V^\a$ is almost a  graph with respect to the $x_n$ variable (see Lemma~\ref{lem Lip}),  for all $t\ge \az^{-1}$ and for $\ez$ small enough we have
$$
\mathcal H^{n-1}((E)_t \cap V^\a )\le \mathcal H^{n-1}((E)_{\az^{-1}} \cap V^\a)+C\ez \leq \frac 34 \mathcal H^{n-1}((V^\a)_{\az^{-1}}).
$$ 
Hence, since
$$\mathcal H^{n-2}((\partial^* E\cap V^\a)_t)=\mathcal  H^{n-2}(\partial^* (E)_t \cap V^\a_t)$$
for a.e. $t$ (see e.g. \cite[Theorem 18.11]{M2012}), we can apply Lemma \ref{weighted ineq} with $F=(E)_t \cap V^\a$ and $\Omega=(V^\a)_t$ to get
\begin{align*}
\int_{\mathbb R} \int_{  (\partial^* E\cap V^\a)_t } \dist(x,\,\partial V^\a)  \, d\mathcal H^{n-2} \, dt&\geq 
\int_{t \geq \alpha^{-1}} \int_{  (\partial^* E\cap V^\a)_t } \dist(x,\,(\partial V^\a)_t)  \, d\mathcal H^{n-2} \, dt \\
&\ge  c(n,K) \int_{t \geq \alpha^{-1}}   \mathcal H^{n-1}((E)_t \cap V^\a) \, dt\\
&=    c(n,K) \,|(E\setminus K^\a)\cap V^\a|. 
\end{align*}
Since $|E\cap V^\a|=|K^\a\cap V^\a|$ (see Lemma~\ref{Ka of E}) it follows that
$$ |(E \setminus K^\a)\cap V^\a|= \frac 1 2|(E\Delta K^\a)\cap V^\a|,$$
and we conclude that 
$$
\int_{\partial^*  E\cap V^\a } [f(\nu) -\nu\cdot X] \, d\mathcal H^{n-1}\ge c(n,K)\, |(E\cap V^\a)\Delta K^\a|,
$$
as desired.

In the second case, namely when
$$\mathcal H^{n-1}((E)_{\az^{-1}} \cap V^\a )> \frac 1 2 \mathcal H^{n-1}((V^\a)_{\az^{-1}}),$$
we simply apply the argument above to $V^\a\setminus E$, and conclude as before.

\subsection{Estimate on II}
In this section we prove that
\begin{equation}\label{main aim 2}
II\ge -\bigl(C\ez^{1/n}+o(1)\bigr)\,|E\setminus K^\a|,
\end{equation}
where $C=C(n,K)$, and $o(1)$ is a quantity that goes to $0$ as $|\a|\to 0$.

Towards this, we note that, by 
the volume constraint $|E|=|K^\a|$, it follows that 
$$|E\setminus K^\a|=|K^\a\setminus E|.$$
Also, since $\partial E$ and $\partial K$ are $(C\ez^{1/n})$-close (see Lemma \ref{lem:E close K}),
it follows by \eqref{almost identity 2} and \eqref{almost identity 3} that 
$$
|{\rm div} X-(n-1)|\leq C\ez^{1/n}+o(1)\qquad \text{on }(E\Delta K^\a)\setminus \mathcal{C}_T,
$$
and
$$
-C\leq {\rm div} X-(n-1)\leq C\ez^{1/n}+o(1)\qquad \text{on }(E\Delta K^\a)\cap \mathcal{C}_T,
$$
where $C=C(n,K).$
Hence, by the divergence theorem we get
\begin{equation}
\label{eq:II}
\begin{split}
II=\int_{\partial^* E } X\cdot \nu\,d\mathcal H^{n-1}&-\int_{\partial^* K^\a} X\cdot \nu\,d\mathcal H^{n-1}\\
&=\int_{E\setminus K^\a } {\rm div}X \,d\mathcal H^{n-1}-\int_{  K^\a\setminus E} {\rm div}X \,d\mathcal H^{n-1}\\
&=\int_{E\setminus K^\a } [{\rm div}X-(n-1)] \,d\mathcal H^{n-1}\\
&\qquad -\int_{  K^\a\setminus E} [{\rm div}X -(n-1)] \,d\mathcal H^{n-1}\\
& \ge \int_{  (E\setminus K^\a)\cap \mathcal{C}_T} [{\rm div}X-(n-1)] \,d\mathcal H^{n-1}-\bigl(C\ez^{1/n}+o(1)\bigr)|K^\a\Delta E| \\
&\ge -C|(E\setminus K^\a)\cap \mathcal C_T| - \bigl(C\ez^{1/n}+o(1)\bigr)|K^\a\Delta E|,
\end{split}
\end{equation}
Thus, to conclude the proof, we need to show that $|(E\setminus K^\a)\cap \mathcal C_T|$ is small compared to $|K^\a\Delta E|$.

To this aim, we write 
$$K^\a_{1+r}=(1+r)K^\a$$
 and note that
$$\partial K_{1+r}^\a=\{x\in \mathbb R^n \colon f_*^\a(x)=1+r\}.$$
Then, since $|\nabla f_*^\a|\leq C(n,K)$ and $f(\nu)\geq c(n,K)>0$, by the classical coarea formula (see for instance \cite[Section 3.4.4, Proposition 3]{EG1992}) we get
\begin{multline*}
|E\setminus K^\a|=\int_{0}^{\infty}\int_{\partial^* K_{1+r}^\a} \frac 1 {|\nabla f_*^\a|}\chi_{E\setminus K^\a}\,d\mathcal H^{n-1}\,dr\\
\ge c(n,K)\int_{0}^{\infty}  \mathcal H^{n-1}(E\cap \partial K_{1+r}^\a) \,dr\ge c(n,K)\int_{0}^{\infty}  \int_{E\cap \partial^* K_{1+r}^\a} f(\nu)\,d\mathcal H^{n-1} \,dr.
\end{multline*}
Thus, setting for simplicity $\v:=|E\setminus K^\a|$, we get
$$
\v\ge c(n,K)\int_{0}^{\infty}  \int_{E\cap \partial^* K_{1+r}^\a} f(\nu)\,d\mathcal H^{n-1} \,dr
\ge c(n,K)\int_{0}^{M\v}  \int_{E\cap \partial^* K_{1+r}^\a} f(\nu)\,d\mathcal H^{n-1} \,dr,
$$
for some large constant $M$ to be determined. Then by the mean value theorem, there exists $r_0\in[0,\,M\v]$ so that
$$\int_{E\cap \partial^* K_{1+r_0}^\a} f(\nu)\, d\mathcal H^{n-1}\le \frac 1 M. $$
Since $\partial E$ is almost a Lipschitz graph by Lemma~\ref{lem Lip}, we conclude that
\begin{equation}\label{small perimeter}
\int_{E\cap \partial^* K_{1+r}^\a} f(\nu)\, d\mathcal H^{n-1}\le \frac 1 M+C\ez\qquad \forall\, r \geq r_0.
\end{equation}
Moreover, recalling that $\mathcal C_T$ is the cone over a set $T\subset \partial K^\a$ satisfying $\mathcal H^{n-1}(T)=o(1)$, we further have
\begin{equation}\label{small volume}
|(E\setminus K^\a)\cap \mathcal C_T\cap K_{1+r_0}^\a|\le C(K)r_0 \mathcal H^{n-1}(T\cap K^\a)\le o(1)M\v. 
\end{equation}

We now claim that, if $M=M(n,K)$ is sufficiently large, then
\begin{equation}\label{quantify}
E\subset K^\a_{1+r_0}
\end{equation}
(recall that $r_0 $ depends on $M$). 
 Towards this, we first show the following lemma. 
\begin{lem}\label{lower estimate f}
Let  $V$  be a cone for $\partial K$, and denote by $x_0$ the barycenter of the face $\partial K\cap V$. Then, for any $\nu\in \mathbb R^{n}$ we have
$$f(\nu)\ge \nu\cdot x_0+c(n,K)\left|\nu'\right|, $$
where $\nu'$ denotes the projection of $\nu$  onto the hyperplane parallel to $\partial K\cap V$.
\end{lem}
\begin{proof}
By the definition of $f$ we have
$$f(\nu)\ge \sup_{x\in \partial K\cap \partial V} \nu\cdot x\ge \nu\cdot x_0+\sup_{x\in \partial K\cap \partial V} \nu\cdot (x-x_0).$$
Notice that
$x-x_0$ is parallel to the fact $\partial K\cap V$. Thus, by choosing $x\in \partial K\cap \partial V$
so that $x-x_0$ is parallel to $\nu'$,
and noticing that  $|x-x_0|\geq c(n,K)>0$ (since $x_0$ is the barycenter of $\partial K\cap V$), we  obtain
$$f(\nu)\ge\nu\cdot x_0+ \nu\cdot (x-x_0) \ge \nu\cdot x_0+c(n,K) |\nu'|,$$
as desired.
\end{proof}

Now, we fix one of the cones $V^\a=V_i^\a$, and  apply Lemma~\ref{lower estimate f} to $V=V_i$.
Up to a change of variables we can assume that the normal of $\partial K\cap V$ is given by $e_n$.
Hence, denoting by $\nu'$ the projection of $\nu$  onto the first $(n-1)$-variables (i.e., $\nu =(\nu',\nu_n) \in \R^{n-1}\times \R$), we have
\begin{multline*}
\int_{(\partial^* E\cap V^\a)\setminus K^\a_{1+r_0}} f(\nu)\,d\mathcal H^{n-1}\\ \ge \int_{(\partial^* E\cap V^\a)\setminus K^\a_{1+r_0}} \nu\cdot x_0 \,d\mathcal H^{n-1}+ c(n,K)\int_{(\partial^* E\cap V^\a)\setminus K^\a_{1+r_0}}|\nu'|\,d\mathcal H^{n-1}.
\end{multline*}
\begin{figure}[ht]
	\centering
	\def\svgwidth{300 pt}
\begingroup%
  \makeatletter%
  \providecommand\color[2][]{%
    \errmessage{(Inkscape) Color is used for the text in Inkscape, but the package 'color.sty' is not loaded}%
    \renewcommand\color[2][]{}%
  }%
  \providecommand\transparent[1]{%
    \errmessage{(Inkscape) Transparency is used (non-zero) for the text in Inkscape, but the package 'transparent.sty' is not loaded}%
    \renewcommand\transparent[1]{}%
  }%
  \providecommand\rotatebox[2]{#2}%
  \newcommand*\fsize{\dimexpr\f@size pt\relax}%
  \newcommand*\lineheight[1]{\fontsize{\fsize}{#1\fsize}\selectfont}%
  \ifx\svgwidth\undefined%
    \setlength{\unitlength}{209.76377953bp}%
    \ifx\svgscale\undefined%
      \relax%
    \else%
      \setlength{\unitlength}{\unitlength * \real{\svgscale}}%
    \fi%
  \else%
    \setlength{\unitlength}{\svgwidth}%
  \fi%
  \global\let\svgwidth\undefined%
  \global\let\svgscale\undefined%
  \makeatother%
  \begin{picture}(1,0.40540541)%
    \lineheight{1}%
    \setlength\tabcolsep{0pt}%
    \put(0.09078206,0.03913019){\color[rgb]{0,0,0}\makebox(0,0)[lt]{\lineheight{1.25}\smash{\begin{tabular}[t]{l}$O$\end{tabular}}}}%
    \put(0,0){\includegraphics[width=\unitlength,page=1]{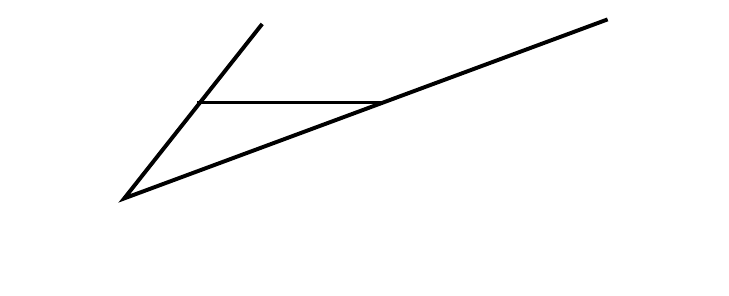}}%
    \put(0.7199956,0.2879982){\color[rgb]{0,0,0}\makebox(0,0)[lt]{\lineheight{1.25}\smash{\begin{tabular}[t]{l}$V$\end{tabular}}}}%
    \put(0,0){\includegraphics[width=\unitlength,page=2]{cone.pdf}}%
    \put(0.35530214,0.18156407){\color[rgb]{0,0,0}\makebox(0,0)[lt]{\lineheight{1.25}\smash{\begin{tabular}[t]{l}$x_0$\end{tabular}}}}%
    \put(0,0){\includegraphics[width=\unitlength,page=3]{cone.pdf}}%
    \put(0.07826038,0.18156407){\color[rgb]{0,0,0}\makebox(0,0)[lt]{\lineheight{1.25}\smash{\begin{tabular}[t]{l}$f(e_n)$\end{tabular}}}}%
  \end{picture}%
\endgroup%

	\caption{ The vector $x_0$ and the length of $f(e_n)$ are shown in the figure. The dark grey part represents $K^\a\cap V^{\a}$. 
In our proof, we apply the divergence theorem to the constant vector field $x_0$ inside the set $(E\cap V^\a)\setminus K^\a_{1+r_0}$. }
	\label{fig:cone}
\end{figure}

For the first term, by the divergence theorem applied to the constant vector field $x_0$ inside the set
$$
(E\cap V^\a)\setminus K^\a_{1+r_0},
$$
we get
$$
\int_{(\partial^* E\cap V^\a)\setminus K^\a_{1+r_0}} \nu\cdot x_0 \,d\mathcal H^{n-1}=\int_{\partial K^\a_{1+r_0}\cap V^\a \cap E} e_n\cdot x_0 \,d\mathcal H^{n-1}- \int_{(\partial V^\a \cap E)\setminus K^\a_{1+r_0}} \nu\cdot x_0 \,d\mathcal H^{n-1}.
$$
Note that, since $x_0 \in V^\a$ (for $|\a|\ll 1$) and $V^\a$ is a convex cone, it follows that $\nu\cdot x_0\leq 0$ on $(\partial V^\a \cap E)\setminus K^\a_{1+r_0}$. Also, $x_0\cdot e_n=f(e_n)=f(\nu)$ on $\partial K^\a_{1+r_0}$ (see Figure \ref{fig:cone}). Hence
$$
\int_{(\partial^* E\cap V^\a)\setminus K^\a_{1+r_0}} \nu\cdot x_0 \,d\mathcal H^{n-1}\geq \int_{\partial K^\a_{1+r_0}\cap V^\a \cap E} f(\nu) \,d\mathcal H^{n-1}.
$$

For the second term, we apply first the coarea formula for rectifiable sets (see e.g. \cite[Theorem 18.8]{M2012}) to get
$$
\int_{(\partial^* E\cap V_i^\a)\setminus K^\a_{1+r_0}}|\nu'|\,d\mathcal H^{n-1}=\int_{r_0}^{\infty} \int_{\partial^* E\cap V_i^\a \cap \partial^* K^\a_{1+r}}\, d\mathcal H^{n-2}\,dr.
$$
Then, provided $M^{-1}$ and $\ez$ are small enough, thanks to \eqref{small perimeter} we can apply the relative isoperimetric inequality to $E\cap V^\a\cap \partial K^\a_{1+r}$ inside the convex set $\partial K^\a_{1+r}$ for $r \geq r_0$  to obtain
\begin{align*}
\int_{r_0}^{\infty} \int_{\partial^* E\cap V^\a \cap \partial^* K^\a_{1+r}}\, d\mathcal H^{n-2}\,dr &\ge c(n,K)    \int_{r_0}^{\infty}\left(\mathcal H^{n-1}(E\cap V^\a \cap \partial^* K^\a_{1+r})\right)^{\frac {n-2} {n-1}}\, dr\\
&\ge c(n,K)\left(M^{-1}+C\ez\right)^{-\frac 1 {n-1}}  \int_{r_0}^{\infty}\mathcal H^{n-1}(E\cap V^\a \cap \partial^* K^\a_{1+r})\, dr\\
&= c(n,K) \left(M^{-1}+C\ez\right)^{-\frac 1 {n-1}} |(E\cap V^\a)\setminus K^\a_{1+r_0}|,
\end{align*} 
where the second inequality follows by \eqref{small perimeter}.

Combining all the previous estimates, we proved that
\begin{multline*}
\int_{(\partial^* E\cap V^\a)\setminus K^\a_{1+r_0}} f(\nu)\,d\mathcal H^{n-1}- \int_{\partial K^\a_{1+r_0}\cap V^\a \cap E} f(\nu) \,d\mathcal H^{n-1}\\
\geq c(n,K) \left(M^{-1}+C\ez\right)^{-\frac 1 {n-1}} |(E\cap V^\a)\setminus K^\a_{1+r_0}|,
\end{multline*}
and by summing this inequality over all cones $V^\a_i$ we conclude that 
\begin{equation}\label{comparison}
\int_{\partial^* E\setminus K^\a_{1+r_0}} f(\nu)\, d\mathcal H^{n-1}-\int_{\partial^* K^\a_{1+r_0} \cap E} f(\nu)\, d\mathcal H^{n-1}\ge c(n,K) \left(M^{-1}+C\ez\right)^{-\frac 1 {n-1}} |E\setminus K^\a_{1+r_0}|.
\end{equation} 

On the other hand, if we test the 
$(\ez,\,R)$-minimality of $E$ against the set
$$G=(1+\lambda)[E\cap K^\a_{1+r_0}],$$ where 
 $\lambda$ is chosen so that
 $$\left|(1+\lambda)[E\cap K^\a_{1+r_0}]\right|=|E|$$
(note that this set is admissible for $\ez$, and hence $|E\Delta K|$, small enough), we obtain  
$$\mathscr F(E)\le (1+\lambda)^{n-1}\mathscr F(E\cap K^\a_{1+r_0})+C\lambda \ez \leq \mathscr F(E\cap K^\a_{1+r_0}) +C(n,K)\lambda.$$
Note that, by the definition of $\lambda$ and Lemma~\ref{lem:E close K}, we easily get the bound 
$$\lambda \leq C(n,K)|E\setminus K_{1+r_0}^\a|,$$
and hence   deduce that 
$$\mathscr F(E)\le \mathscr F(E\cap K^\a_{1+r_0}) +C(n,K)|E\setminus K_{1+r_0}^\a|.$$
Combining this bound with \eqref{comparison}, we conclude that 
\begin{align*}
c(n,K)\left(M^{-1}+C\ez\right)^{-\frac 1 {n-1}} |E\setminus K^\a_{1+r_0}|&\le \int_{\partial^* E\setminus K^\a_{1+r_0}} f(\nu)\, d\mathcal H^{n-1}-\int_{\partial^* K^\a_{1+r_0} \cap E} f(\nu)\, d\mathcal H^{n-1}\\
&= \mathscr F(E)- \mathscr F(E\cap K^\a_{1+r_0})\le C(n,K) |E\setminus K^\a_{1+r_0}|.
\end{align*}
Thus, for $M=M(n,K)>0$ sufficiently large and $\ez \leq \ez_0(n,K)$ small enough, we conclude that $|E\setminus K^\a_{1+r_0}|=0$,
and \eqref{quantify} follows. Then, \eqref{main aim 2} is an immediate consequence of \eqref{eq:II}, \eqref{small volume}, and \eqref{quantify}.

\subsection{Conclusion}

Since $E$ is an $(\ez,R)$-minimizer with $R\geq n+1$, for $\ez \ll 1$ we 
have that $|\a|\ll 1$, and therefore
$$\mathscr F(E)\le \mathscr F(K^\a)+\ez|E\Delta K^\a|. $$
On the other hand, combining \eqref{A+B}, \eqref{main aim 1}, and \eqref{main aim 2}, we get
$$\mathscr F(E)- \mathscr F(K^\a)\ge c(n,K)|E\Delta K^\a|.$$
Choosing $\ez$ sufficiently small we conclude that $|E\Delta K^\a|=0$, which proves Theorem~\ref{main thm}.

\section{Proof of Theorem~\ref{characterization} and ~\ref{stability}}
\begin{proof}[Proof of Theorem~\ref{characterization}]
	Let $K^\a\in \mathscr C(K)$ with small $|\a|$, and let $\ez_0$ be the constant in Theorem~\ref{main thm}. Then by Theorem~\ref{main thm}, the following variational problem,
\begin{equation}\label{eps R}
\min\left\{\mathscr F (E)+\ez|E\Delta K^\a|\colon |E|=|K^\a| \right\}
\end{equation}
	among all sets of finite perimeter $E$ with $\dist(x,\,K^\a)\le R$ for any $x\in E$, has a solution $K^{\a'}$, where $\ez\le \ez_0$.
	Then we have
	$$\mathscr F (K^{\a'})+\ez|K^{\a'}\Delta K^\a|\le \mathscr F(K^\a).$$
Since
	$$|K^{\a'}\Delta K^\a|\ge c(n,K)|\a'-\a|$$
	(this follows easily by the argument used in the proof of Lemma \ref{lem:F C1}),
we conclude that
	$$c(n,K)\ez|\a'-\a|\le  o(1)|\a -\a'|,$$
which proves that $\a'=\a$ for $|\a|+|\a'| $ sufficiently small.
Therefore, $K^\a$ is the unique minimizer of  \eqref{eps R}, and then the theorem follows. 
\end{proof}

\begin{proof}[Proof of Theorem~\ref{stability}]
	Choose $\sigma_0$ small so that we can apply Lemma~\ref{Ka of E} to obtain a set $K^\a$ for $E$. We apply the idea of \cite{FJ2014}, i.e. the selection principle.
	
	Towards this let us assume that the conclusion of the theorem fails.
	Then there exist a sequence of sets of finite perimeter $E_k$, and vectors $\a_k \in \R^N$ with $|\a_k|$ small, such that $|E_k\Delta K|\to 0$, $|E_k\cap V_i^{\a_k}|=|K^{\a_k}\cap V_i^{\a_k}|$ for any $1\le i\le N$, $\mathscr F(E_k )- \mathscr F(K^{\a_k} ) \to 0$,
	and
	\begin{equation}\label{assump} 
 0<\mathscr F(E_k)-  \mathscr F(K^{\a_k})< \frac{\lambda}{2}|E_k\Delta K^{\a_k}|
	\end{equation}
	for some $\lambda>0$ small enough (the smallness will be fixed later).
	Notice that, since $|E_k\Delta K|\to 0$,  we have $|\a_k|\to 0$.
	
%

Consider the following variation problem:
\begin{equation}\label{sp}
\min\left\{\mathscr F(E)+\lambda\big||E\Delta K^{\a_k}|-|E_k\Delta K^{\a_k}|\big|+C_0\lambda \sum_{1\le i\le N} \big| |E\cap V_i^{\a_k}| -|K^{\a_k}\cap V_i^{\a_k}|\big|\colon |E|=|K^{\a_k}| \right\}
\end{equation}
with $C_0=C_0(n,K)>0$ to be determined.  Since the functional involved in the problem is lower semicontinuous with respect to the $L^1$-convergence, there exists a minimizer $F_k$ among sets of finite perimeter (cf. \cite[Chapter 12]{M2012}). 
Thus, for any set $G\subset \mathbb R^n$ with $|G|=|F_k|$ we have
\begin{align*}
\mathscr F(F_k)&\le \mathscr F(G)+\lambda\Big(\big||G\Delta K^{\a_k}|-|E_k\Delta K^{\a_k}|\big|-\big||F_k\Delta K^{\a_k}|-|E_k\Delta K^{\a_k}|\big|\Big)\\
&\qquad +C_0\lambda \sum_{1\le i\le N}\Big( \big| |G\cap V_i^{\a_k}| -|K^{\a_k}\cap V_i^{\a_k}|\big|
- \big| |F_k\cap V_i^{\a_k}| -|K^{\a_k}\cap V_i^{\a_k}|\big|\Big) \\
 &\le  \mathscr F(G)+\lambda  \big||G\Delta K^{\a_k}|-|F_k\Delta K^{\a_k}|\big| +C_0\lambda \sum_{1\le i\le N} \big| |G\cap V_i^{\a_k}|  
- |F_k\cap V_i^{\a_k}| \big|\\
& \le  \mathscr F(G)+\lambda  | F_k\Delta G |+C_0\lambda \sum_{1\le i\le N} |(F_k\Delta G)\cap V_i^{\a_k}|= \mathscr F(G)+(1+C_0)\lambda  | F_k\Delta G |.
\end{align*}
This proves that $F_k$ is a $\bigl((1+C_0)\lambda\bigr)$-minimizer, and by choosing $\lambda$ small enough so that  $(1+C_0)\lambda \le \ez_0$ as in Theorem~\ref{main thm}, there exists $\a_k' \in \R^N$ such that $F_k=K^{\a_k'}$.

Now since $K^{\a_k'}$ is the minimizer of \eqref{sp} we have
$$
\mathscr F(K^{\a_k'})+\lambda \big||K^{\a_k'}\Delta K^{\a_k}|-|E_k\Delta K^{\a_k}|\big|\le \mathscr F(E_k).
$$
Thus, noticing that $\mathscr F(E_k)\to \mathscr F(K)$, $E_k\to K$, and $|\a_k|\to 0$, it follows by the inequality above that $|\a_k'|\to 0$.

Therefore, by Lemma~\ref{lem:F C1}, and by testing the minimality of $K^{\a_k'}$ in \eqref{sp} against $K^{\a_k},$  we have
\begin{align*}
\omega(|\a_k|+|\a'_k|) |\a'_k-\a_k| 
&\geq \mathscr F(K^{\a_k})-\mathscr F(K^{\a_k'})\\
&\geq \lambda \big||K^{\a_k'}\Delta K^{\a_k}|-|E_k\Delta K^{\a_k}|\big| -\lambda|E_k\Delta K^{\a_k}| \\
&\qquad +C_0\lambda \sum_{1\le i\le N} \big| |K^{\a_k'}\cap V_i^{\a_k}| -|K^{\a_k}\cap V_i^{\a_k}|\big|\\
&\geq C_0\lambda \sum_{1\le i\le N} \big| |K^{\a_k'}\cap V_i^{\a_k}| -|K^{\a_k}\cap V_i^{\a_k}|\big| -\lambda |K^{\a_k'}\Delta K^{\a_k}|.
\end{align*}
As in the proof of Theorem \ref{characterization} one can note that 
$$
\sum_{1\le i\le N} \big| |K^{\a_k'}\cap V_i^{\a_k}| -|K^{\a_k}\cap V_i^{\a_k}|\big|\geq c(n,K)|\a'_k-\a_k|,\qquad 
|K^{\a_k'}\Delta K^{\a_k}|\leq C(n,K)|\a_k'-\a_k|.
$$
Hence, choosing $C_0=C_0(n,K)$ sufficiently large, we deduce
$$
\omega(|\a_k|+|\a'_k|) |\a'_k-\a_k| 
\geq c(n,K)\lambda |\a_k'-\a_k|,
$$
which proves that $|\a_k|=|\a'_k|$ for $k \gg 1$.
Hence, testing the minimality of $K^{\a_k}$ in \eqref{sp} against $E_k,$ and recalling \eqref{assump},
 we get
$$\mathscr F(K^{\a_k})+\lambda |E_k\Delta K^{\a_k}| \le \mathscr F(E_k)\le \mathscr F(K^{\a_k})+\frac{\lambda}{2} |E_k\Delta K^{\a_k}|,$$
which implies $|E_k\Delta K^{\a_k}|=0$ and $\mathscr F(E_k)=\mathscr F(K^{\a_k})$. This contradicts \eqref{assump}, and 
the theorem follows.

\end{proof}

\appendix
\section{Technical results}
In this appendix we prove some technical results used in the paper.

First of all, we prove a  weighted 
relative isoperimetric inequality.
\begin{lem}\label{weighted ineq}
	Let $\Omega\subset \mathbb R^d$ be a Lipschitz domain, and $F\subset \Omega$ a set of finite perimeter. Then there exists a constant $C=C(d,\,\Omega)>0$ such that 
	$$ |F\cap \Omega| \le C \int_{\partial^* F\cap \Omega}\dist(x,\,\partial \Omega)\, d\mathcal H^{n-1}.,$$
	whenever $|F\cap \Omega|\le \frac 3 4 |\Omega|.$
\end{lem}

\begin{proof}
	Let $u=\chi_{F\cap \Omega}$ and $u_k$ be a sequence of smooth functions approximating $u$ strongly in $L^1$, $\|Du_k\| \rightharpoonup \|Du\|$ weakly$^*$ as  measures; see e.g. \cite[Theorems 2\&3, Chapter 5.2]{EG1992} and \cite[Chapter 5.4]{EG1992}. 

Then, by applying the weighted Poincar\'e-type inequality from \cite{BS1988} to the functions $u_k$, when $k$ is large enough, it
yields
$$
\int_{\Omega} |u_k|\, dx\le  C \int_{\Omega}  |Du_k| \dist(x,\,\partial \Omega)\, dx,
$$
where $C=C(d,\,\Omega).$
Letting $k \to \infty$ yields the result.
\end{proof}

We now state a Lipschitz regularity result for $\ez$-minimizers.
\begin{lem}
\label{lem Lip}
For any $i=1,\ldots,N$, let $H_i$ denotes the hyperplane containing $\partial K\cap V_i$.
Let $E$ be an $(\ez,n+1)$-minimizer of $\mathscr F$ with $|E|=|K|$. There exist $\bar \ez=\bar \ez(n,K)>0$ and $L=L(n,K)>0$ such that if $\ez \leq \bar\ez$ then the following holds:
For any $i=1,\ldots,N$ there exists a neighborhood $\mathcal U_i$ of $\partial K\cap V_i$ such that, up to a translation of $E$,
$$\partial E\subset \cup_{1\leq i\leq N}\mathcal U_i$$
and
$$
\text{$(\partial E\cap \mathcal U_i)\setminus \Gamma_i$ is a $L$-Lipschitz graph with respect to $H_i$,}
$$ where $\Gamma_i\subset \partial E$ satisfies $\mathcal H^{n-1}(\Gamma_i)\leq C(n,K)\ez$.
\end{lem}
\begin{proof}
Thanks to Lemma \ref{lem:E close K}, up to a translation 
we have that $\partial E$ is uniformly close to $\partial K$.
This allows us to apply
 \cite[Proposition 4.6]{ANP2002} and deduce that, for $\ez$ sufficiently small, we can cover almost all the boundary of $E$ with uniformly Lipschitz graphs.
\end{proof}

\end{document}